\documentclass[11pt]{amsart}

\usepackage{amsmath,amssymb,amsthm}
\usepackage{graphicx}
\usepackage{cite}
\usepackage[usenames, dvipsnames]{color}
\usepackage{hyperref}

\normalsize

\def\dfrac#1#2{\lower0.15ex\hbox{\large$\frac{#1}{#2}$}}

\allowdisplaybreaks

\def\eps{\varepsilon}

\def\C{\mathbb{C}}

\def\F{\mathbb{F}}

\def\E{\mathbb{E}}

\DeclareMathOperator\Af{A4}
\DeclareMathOperator\Ai{A5}
\newcommand{\vect}[1]{\boldsymbol{#1}}

\newtheorem{firstthm}{Proposition}[section]
\newtheorem{thm}[firstthm]{Theorem}
\newtheorem{prop}[firstthm]{Proposition}

\newtheorem{lemma}[firstthm]{Lemma}

\newtheorem{ques}[firstthm]{Question} 
\newtheorem*{alon-mainthm}{Theorem \ref{t:main-alon}}

\theoremstyle{definition}

\newtheorem{remark}[firstthm]{Remark} 
 
\newtheorem{defn}[firstthm]{Definition}

\title{On a question of Alon}

\author{Daniel Altman}
\address{University of Oxford, Mathematical Institute, Radcliffe Observatory Quarter, Woodstock Rd, Oxford, OX2 6GG, United Kingdom}
\email{daniel.h.altman@gmail.com}

\begin{document}

\maketitle

\begin{abstract}
A system of linear equations in $\F_p^n$ is \textit{common} if every two-colouring of $\F_p^n$ yields at least as many monochromatic solutions as a random two-colouring, asymptotically as $n \to \infty$. By analogy to the graph-theoretic setting, Alon has asked whether any (non-Sidorenko) system of linear equations can be made uncommon by adding sufficiently many free variables. Fox, Pham and Zhao answered this question in the affirmative among systems which consist of a single equation. We answer Alon's question in the negative.

We also observe that the property of remaining common despite that addition of arbitrarily many free variables is closely related to a notion of commonness in which one replaces the arithmetic mean of the number of monochromatic solutions with the geometric mean, and furthermore resolve questions of Kam\v cev--Liebenau--Morrison.
\end{abstract}

\section{Introduction}

In graph theory, a graph $H$ is called \textit{common} if every two-colouring of $K_n$ contains at least at many monochromatic copies of $H$ as a random two-colouring does (in the limit $n\to \infty$). Goodman \cite{G59} proved that $K_3$ is common. Erd\H os conjectured that $K_4$ is common  \cite{E62} and subsequently Burr and Rosta conjectured \cite{BR80} that every graph is common. These conjectures were disproved by Sidorenko \cite{Sid89} and Thomason \cite{T89} who demonstrated respectively that a triangle with a pendant edge and $K_4$ are in fact not common. The classification of common graphs remains very much open to this day.

One observes that if $H$ satisfies the stronger property that for every large graph $G$, among all subsets of $G$ of fixed density, a random subset of $G$ contains the fewest copies of $H$ (in the limit $|V(G)| \to \infty$), then $H$ is certainly common.  Graphs satisfying this stronger property are called \textit{Sidorenko}. Conjecturally \cite{S93}, all bipartite graphs are Sidorenko; this conjecture remains wide open (despite the resolution of various special cases). Certainly, since Sidorenko's conjecture is not known to be false, all known instances of uncommon graphs are not bipartite. In fact, \textit{every} non-bipartite graph can be made uncommon by the addition of sufficiently many pendant edges (an edge $e$ is pendant to a graph $H$ if $|e \cap V(H)| = 1$).

\begin{thm}[{\cite{JST96}}]
If $H$ is a non-bipartite graph with $m\geq 3$ vertices, then there is a positive integer $l_0(m)$ such that any graph obtained by successively adding at least $l_0$ pendant edges to $H$ is uncommon.
\end{thm}

Alon has asked whether the analogous result is true in the arithmetic setting. A system of linear equations in $\F_p^n$ is \textit{common} if every two-colouring of $\F_p^n$ yields at least as many monochromatic solutions as a random two-colouring, asymptotically as $n \to \infty$. A system of linear equations in $\F_p^n$ is \textit{Sidorenko} if among all subsets $S$ of $\F_p^n$ of fixed density, the number of solutions in $S$ is minimised when $S$ is a random set, asymptotically as $n \to \infty$. It is clear that if a linear system is Sidorenko, then it is common.  Saad and Wolf initiated the study of these properties in the arithmetic setting \cite{SW17}. For both the Sidorenko and common properties, necessary \cite{FPZ21} and sufficient \cite{SW17} conditions for systems containing a single equation are known (in fact, a single equation in an even number of variables is Sidorenko if and only if it is common, and an equation in an odd number of variables is necessarily common but not Sidorenko). The classification problems for systems containing more than one equation remain wide open, despite recent interest and the resolution of special cases; see \cite{FPZ21}, \cite{KLMCom}, \cite{KLMSid}, \cite{Vab}, \cite{V4ap}.

It transpires that adding a free variable to a system of linear equations (equivalently, a free dimension to the solution space) corresponds to the addition of a pendant edge in the graph-theoretic setting. In those cases when one can transfer between the graph-theoretic and arithmetic settings via Cayley graph constructions, these operations are equivalent. See \cite[Example 4.2]{SW17} for an instance of this correspondence and elsewhere in that paper for a brief discussion of the correspondence in general. For $l \geq 0$, let $\Psi^{(l)}$ be the linear system attained by adding $l$ free variables to $\Psi$.

\begin{ques}[Alon, {\cite[Question 4.1]{SW17}}]\label{q:alon}
Is it true that for all non-Sidorenko $\Psi$, there exists $l_0$ such that for all $l\geq l_0$ we have $\Psi^{(l)}$ is uncommon?
\end{ques}

For example, it is known that three term arithmetic progressions, whose solution space is parameterised $\Psi = (x,x+d,x+2d)$, are common. Is it true that the system given by $\Psi^{(l)} = (x,x+d,x+2d,y_1,\ldots, y_l)$ is uncommon for all $l$ sufficiently large? The answer is yes by the following theorem of Fox--Pham--Zhao. 

\begin{thm}[{\cite[Theorem 1.5]{FPZ21}}]
The answer to Question 1.2 is yes among all $\Psi$ whose image has codimension one. 
\end{thm}

In light of Question \ref{q:alon} and for convenience in this document, we will refer to systems $\Psi$ with the property that $\Psi^{(l)}$ is \textit{common} for all sufficiently large $l$ as \textit{Alon}. We have mentioned already that Sidorenko implies common; in fact it is not difficult to see that Sidorenko implies Alon. In this language, the above question essentially (i.e. under the assumption that linear systems cannot oscillate infinitely often between common and uncommon under the addition of free variables) asks whether Sidorenko is equivalent to Alon. This is true in the graph-theoretic setting under Sidorenko's conjecture. It turns out that it is not true in the arithmetic setting.

\begin{thm}\label{t:main-alon}
There is a system of linear equations which is Alon but not Sidorenko, so the answer to Question \ref{q:alon} is no. 
\end{thm}

The example which demonstrates the negative answer to Question \ref{q:alon} is the following rank 2 system of equations in 9 variables:
\begin{align*}
x_1 - x_2 + x_3 -x_4 &= 0\\
x_5 -x_6 + x_7 - x_8 + x_9 &=0.
\end{align*} 

Theorem \ref{t:main-alon} is proven in Section \ref{s:alon-proof}. As an intermediate lemma in the proof, we provide a lower bound on the \textit{geometric} mean of the number of monochromatic solutions to $\Phi$. We show in Section \ref{s:geomcom} that a lower bound of this type is in fact also necessary, and indeed that a slightly weaker version of the Alon property is {equivalent} to a geometric notion of commonness in which one asks that the geometric mean of the number of monochromatic solutions is at least that of a random two-colouring. This is proven in Proposition \ref{p:main}. The intended utility is that geometric commonness is oftentimes more easily studied than the Alon property.

It transpires that $\Phi$ also exhibits a negative answer to  a question of Kam\v cev--Liebenau--Morrison. Recall that a system is said to be \textit{translation-invariant} if it is solved by setting $x_1=x_2=x_3 =\cdots$. If a system is not translation invariant, then the set of elements of $\F_p^n$ whose first entry is equal to one has positive density but no solutions, so any system which is not translation-invariant cannot be Sidorenko. The only known examples of systems which are not translation invariant and are common are systems of rank one (i.e. systems comprising a single equation), whereupon commonness follows trivially from a cancellation which does not occur for systems of two or more equations. This inspired the following question of Kam\v cev--Liebenau--Morrison.

\begin{ques}\label{q:klm5}\cite[Question 5.6]{KLMSid} Does there exist a system of rank at least two which is common, but not translation-invariant?
\end{ques}

\begin{thm}\label{t:klm5}
The answer to Question \ref{q:klm5} is yes.
\end{thm}

The system $\Phi$ has rank two and is not translation invariant. The proof of Theorem~\ref{t:klm5} is completed in Lemma \ref{l:phi-com}, which shows that $\Phi$ is common. Since uploading this document to the arXiv, the author has been made aware that Question \ref{q:klm5} has been independently resolved in $\F_2^n$ by Kr\'al'--Lamaison--Pach in forthcoming work. See \cite{KLP22}.

Finally, we also take this opportunity to address the following question of Kam\v cev--Liebenau--Morrison.

\begin{ques}\cite[Question 6.3]{KLMCom}\label{q:klm6}
Suppose that $\Psi$ is a linear system such that there is a set $A \subset \F_p^n$ such that the
density of monochromatic solutions in $(A, A^c)$ is less than $\alpha^t + (1-\alpha)^t$, where $\alpha = |A|/p^n$. Is $\Psi$ necessarily uncommon if one restricts one's attention to sets of density roughly $1/2$?
\end{ques}
For a suitable interpretation of `roughly', the answer to  \cite[Question 6.3]{KLMCom} is no, even with the stronger demand that the density of monochromatic solutions in $(A,A^c)$ is less than $2^{1-t}$. To see why, we will need to introduce some notation.

\vspace{0.5cm}

\textbf{Acknowledgments.} The author is indebted to Ben Green for valuable feedback on earlier versions of this document, and for his ongoing support. The author thanks Anita Liebenau for helpful correspondences.

\section{Notation, preliminaries, and an answer to Question \ref{q:klm6}}

Throughout this document we assume that $p$ is an odd prime. 

If $\Psi$ is a system of $m$ linear equations in $t$ variables, described by a $m\times t$ matrix $M_\Psi$ with entries in $\F_p$, and $f: \F_p^n \to \C$, then we define
\[ T_\Psi(f) := \E_{\vect x \in (\F_p^n)^t: M_\Psi \vect x = 0} f(x_1)f(x_2)\cdots f(x_t).\]

In this document we only consider the case of homogeneous linear systems: each equation is of the form $\sum_{i=1}^t a_ix_i=0$. Furthermore, we always assume that $t > m$ and that the system is of full rank. In particular, $\Psi$ has nontrivial solutions. 
As is standard and as we shall use throughout the document, $\E$ is a normalised sum, so here is shorthand for $\frac{1}{\#\{\vect x \in (\F_p^n)^t: M_\Psi \vect x = 0\}}\sum$. One observes that if $A\subset \F_p^n$ and $1_A$ is the characteristic function for the set $A$, then $T_\Psi(1_A)$ is a normalised count for the number of solutions to $\Psi$ in $A$. Sometimes, when the underlying linear system is clear from context, we may just use the notation $T(f)$ for brevity.

We use the Fourier transform on $\F_p^n$. Let $e_p(\cdot)$ be shorthand for $e^{\frac{2\pi i \cdot}{p}}$. For $h \in \F_p^n$, we define \[\hat f(h) := \E_{x \in \F_p^n} f(x)e_p(-h\cdot x),\] whereupon one obtains the Fourier inversion formula: \[f(x) = \sum_{h \in \F_p^n} \hat f(h)e_p(x\cdot h),\]
and Parseval's identity: \[\E_{x \in \F_p^n}|f(x)|^2 = \sum_{h \in \F_p^n}|\hat f(h)|^2.\]

 As is also standard, we will deal with a functional version of the notion of commonness in which one studies the set of functions $f:\F_p^n \to [0,1]$ rather than focusing only on the subset of these which are characteristic functions of sets. 

\begin{defn}\label{d:common}
A system of $t$ linear forms $\Psi$ is \textit{common} if, for all $n \geq 1$ and all $f:\F_p^n \to [0,1]$, we have 
\[T_\Psi(f) + T_\Psi(1-f) \geq 2^{1-t}.\]
\end{defn} 
Clearly if a system is common by the above definition then it is common when one restricts only to characteristic functions of sets. The other direction is recovered by a standard argument in which one constructs a random set $A$ from $f$ by choosing $x \in A$ with probability $f(x)$; we won't dwell on this further but remark that details are worked out for example in \cite{FPZ21}.

Sometimes it is more convenient to describe a linear system by parameterising its solution space. In this case we have equivalently: 
\[ T_\Psi(f) = \E_{y_1,\ldots, y_D \in \F_p^n} f(\psi_1(y_1,\ldots, y_D)) \cdots f(\psi_t(y_1,\ldots, y_D)),\]
where each $\psi_i$ is a linear form mapping $(\F_p^n)^D$ to $\F_p^n$,  $x_i = \psi_i(y_1,\ldots, y_D)$ for each $i$ and  $D = t-m$.

Returning to Question \ref{q:klm6}, we demonstrate a $\Psi$, $f:\F_p^n \to [0,1]$ such that $T_\Psi(f) < 2^{1-t}$, but $T_\Psi(f) \geq 2^{1-t}$ for all $f$ with $\E f = 1/2$. By randomly sampling as discussed above, this yields a negative answer to Question \ref{q:klm6}. 

Let $\Psi$ be the system whose solution space is parameterised by $(x,x+d,x+2d)$, so $t=3, m=1, D=2$ and $\psi_i(x,d) = x+(i-1)d$. Let $\Psi^{(1)}$ be parameterised by $(x, x+d, x+2d, y)$.  It is shown in \cite[Example 4.1]{SW17} that $\Psi^{(1)}$ is uncommon (with a sequence of sets $A_n$ for which $\left||A_n|/p^n-\frac{1}{2}\right| > \eps$ for all $n$ and some explicit $\eps>0$). On the other hand, it is well known that $\Psi$ is common (this goes back to Cameron, Cilleruelo and Serra \cite{CCS07}). For a function $f:\F_p^n \to [0,1]$ with $\alpha:= \E_{x\in \F_p^n}f(x)$, one sees that
\begin{equation}\label{e:one-free} T_{\Psi^{(1)}}(f) + T_{\Psi^{(1)}}(1-f) = \alpha T_\Psi(f) + (1-\alpha)T_\Psi(1-f),
\end{equation}
since we average over the free variable $y$. In fact, in general, we see that 
\begin{equation}\label{e:l-free}
T_{\Psi^{(l)}}(f) + T_{\Psi^{(l)}}(1-f) = \alpha^l T_\Psi(f) + (1-\alpha)^lT_\Psi(1-f);
\end{equation}
this formula will be useful later.

Following on from (\ref{e:one-free}), if $\alpha = 1/2$ then 
\[\inf_{f:\E f = \frac{1}{2}} T_{\Psi^{(1)}}(f) + T_{\Psi^{(1)}}(1-f) = \inf_{f:\E f = \frac{1}{2}} \frac{1}{2}(T_\Psi(f) + T_\Psi(1-f)) \geq 2^{-3},\]
since $\Psi$ is common. Thus $\Psi^{(1)}$ is common among average $1/2$ functions, but is not common in general, answering a suitable interpretation of Question \ref{q:klm6} in the negative.

This argument works more generally for adding free variables to any single equation in a odd number of variables: use \cite[Theorems 1.4, 1.5]{FPZ21}. Furthermore, any system of equations which is common among average $1/2$ functions but not Alon yields counterexamples to Question \ref{q:klm6} in this way.

\section{Proof of Theorem \ref{t:main-alon} and Theorem \ref{t:klm5}}\label{s:alon-proof}

Recall we let $\Phi$ be the following rank 2 system of equations in 9 variables
\begin{align*}
x_1 - x_2 + x_3 -x_4 &= 0\\
x_5 -x_6 + x_7 - x_8 + x_9 &=0.
\end{align*}
Let $\Af$ be shorthand for (the system containing) the first equation in $\Phi$, an additive $4$-tuple. Let $\Ai$ be shorthand for the second equation.

Since $\Phi$ is not translation-invariant and has rank two, to prove Theorem \ref{t:klm5} we just need to show that $\Phi$ is common. 

\begin{lemma}[Proof of Theorem \ref{t:klm5}]\label{l:phi-com}
$\Phi$ is common.
\end{lemma}
\begin{proof}
Let $f:\F_p^n \to [0,1]$ have $\E f = \alpha$ and let $g =  f - \alpha$. We observe that $T_\Phi(f) = T_{\Af}(f)T_{\Ai}(f)$. Furthermore, since $\Af$ has rank one, one sees that $T_{\Af}(\alpha + g) = \alpha^4 + T_{\Af}(g)$, and similarly for $\Ai$. Thus, 
\[ T_\Phi( f) = T_\Phi (\alpha + g) = \alpha^9 + \alpha^5 T_{\Af}(g) + \alpha^4 T_{\Ai}(g) + T_\Phi (g),\]
so one computes
\begin{equation}\label{e:A4A5}
T_\Phi (f) + T_\Phi (1- f) = \alpha^9 + (1-\alpha)^9 + \left( \alpha^5 + (1-\alpha)^5\right)T_{\Af}(g) + \left( \alpha^4 - (1-\alpha)^4\right)T_{\Ai}(g).
\end{equation}
Recall that $T_{\Af}(g)$ is the ($4$th power of the) $\ell^4$-norm of $g$ on the Fourier side so in particular $T_{\Af}(g) \geq 0$. Thus if $\alpha = 1/2$ then we trivially have $T_\Phi (f) + T_\Phi (1-  f)  \geq 1/2^8$. Assume without loss of generality $\alpha < 1/2$. We have $T_{\Af}(g) = \sum_h |\hat g(h)|^4$ and $T_{\Ai}(g) = \sum_h |\hat g(h)|^4\hat g(h)$, and so $|T_{\Ai}(g)|\leq T_{\Af}(g)||\hat g||_\infty \leq T_{\Af}(g)(1-\alpha)$. Thus continuing from (\ref{e:A4A5}),
\[T_\Phi (f) + T_\Phi (1- f) \geq \alpha^9 + (1-\alpha)^9 + T_{\Af}(g) \left(\alpha^5 + \alpha^4(1-\alpha)  \right) \geq \alpha^9 + (1-\alpha)^9 \geq 2^{-8},\]
where the final inequality follows from convexity. 
\end{proof}

Now we proceed to the proof of Theorem \ref{t:main-alon}. Since $\Phi$ is not translation-invariant, it is not Sidorenko. To prove Theorem \ref{t:main-alon} we will show that $\Phi$ is Alon, i.e. $\Phi^{(l)}$ is common for all $l$ sufficiently large. First we compile some properties of $\Phi$. 

Recall that we have assumed that $p$ is odd. 
\begin{lemma}\label{l:phi-prev}
There is a constant $c_0>0$ such that if $f: \F_p^n\to [0,1]$ has $\E f \geq 0.45$ then $T_\Phi(f)\geq c_0$.
\end{lemma}
\begin{proof}
Let $\alpha \in (0,1/2)$ and let $f: \F_p^n \to [0,1]$ have $\E f =\alpha$. Let $g = f-\alpha$. Since $g$ is real-valued, we have $|\hat g(h) | = |\hat g(-h)|$ for all $h$, and so invoking the fact that $\hat g(0)=0$, the fact that the characteristic is not 2, and Parseval's identity, we obtain
\begin{equation}\label{e:g-hat-infty}
2\sup_h |\hat g(h)|^2 \leq \sum_h |\hat g(h)|^2 = \E |g(x)|^2 \leq (1-\alpha)^2.
\end{equation}
Thus $||\hat g||_\infty \leq (1-\alpha)/\sqrt{2}$. By Fourier-inversion as above, it follows that $|T_{\Ai}(g)| \leq \frac{1-\alpha}{\sqrt 2}T_{\Af}(g)$, and by Parseval again $T_{\Af}(g) \leq (1-\alpha)^4/2$. Thus 
\begin{align*}
T_\Phi(f) = T_\Phi(\alpha + g) &=  \alpha^9 +  \alpha^5 T_{\Af}(g) + \alpha^4 T_{\Ai}(g) + T_{\Af}(g)T_{\Ai}(g)\\
&\geq \alpha^9 + T_{\Af}(g) \left(\alpha^5 - \frac{(1-\alpha)\alpha^4}{\sqrt{2}} - \frac{(1-\alpha)^5}{2\sqrt 2}\right).
\end{align*} 
One checks that the polynomial $q(x)= x^5 - \frac{(1-x)x^4}{\sqrt{2}} - \frac{(1-x)^5}{2\sqrt 2}$ is negative for $x \in (0,1/2)$ and so 
\[T_\Phi(f) \geq  \alpha^9 + \frac{(1-\alpha)^4}{2} \left(\alpha^5 - \frac{(1-\alpha)\alpha^4}{\sqrt{2}} - \frac{(1-\alpha)^5}{2\sqrt 2}\right).\]
Defining $\tilde q(x)=x^9 + \frac{1}{2}(1-x)^4\left(x^5-\frac{(1-x)x^4}{\sqrt{2}} - \frac{(1-x)^5}{2\sqrt{2}}\right)$, one checks that $\tilde q(0.45) > 0$, which proves the result when $\E f = 0.45$.  To conclude we observe that $\inf_{f : \E f = \alpha}T_\Phi(f)$ is monotonic in $\alpha$. 
\end{proof} 

Next we show that $\Phi$ possesses a `local Sidorenko' property with respect to the $||\hat \cdot ||_\infty$ norm.

\begin{lemma}\label{l:local-sid}
There is a constant $c_1>0$ such that, upon letting $f: \F_p^n \to [0,1]$,  $\alpha := \E f$ and $g := f - \alpha$,  if $\frac{1}{3} \leq \alpha \leq \frac{2}{3}$ and $||\hat g ||_\infty \leq c_1$ then $T_\Phi(f) \geq \alpha^9$. 
\end{lemma}
\begin{proof}
Let $\alpha \in [\frac{1}{3},\frac{2}{3}]$. For $g:\F_p^n \to [-\alpha, 1-\alpha]$ with $\E_x g(x)=0$ we have, as we saw in Lemma \ref{l:phi-com}, that $T_{\Ai}(g) \leq ||\hat g||_\infty T_{\Af}(g) \leq ||\hat g ||_\infty^3$, where we have used Parseval and the bound $||g||_\infty < 1$. Thus
\begin{align*}
T_\Phi(\alpha+g) &= \alpha^9 + \alpha^5T_{\Af}(g) + \alpha^4T_{\Ai}(g) + T_{\Af}(g)T_{\Ai}(g) \\
&\geq \alpha^9 + T_{\Af}(g)(\alpha^5 - \alpha^4||\hat g||_\infty - ||\hat g||_\infty^3).
\end{align*}
We have that $T_{\Af}(g)\geq 0$ so it remains to argue that there is a constant $c_1$ such that for all $0\leq x \leq c_1$ and all $\frac{1}{3} \leq \alpha \leq  \frac{2}{3}$, we have $\alpha^5 - \alpha^4 x -x^3 \geq 0$. This is easily seen to be true.
\end{proof}

In our final lemma before the proof of Theorem \ref{t:main-alon}, we establish the following inequality.

\begin{lemma}\label{l:phi-geomcom}
There are constants $c_2, c_3, C_4 > 0$ such that upon letting $ f: \F_p^n \to [0,1]$, $\alpha := \E f$ and $g:= f -\alpha$ we have that if $|\alpha - \frac{1}{2}|\leq c_2$ then 
\[ T_\Phi(f) T_\Phi(1-f) \geq 2^{-18} + c_3||\hat g||_\infty^4 - C_4|\alpha - \frac{1}{2}|.\]
\end{lemma}
\begin{proof}
Letting $T_4:=T_{\Af}(g)$ and $T_5 := T_{\Ai}(g)$ for brevity we have
\begin{align*}
T_\Phi(f)T_\Phi(1-f)&=T(\alpha + g)T(1-\alpha -g)\\
&=(\alpha^9 + \alpha^5 T_4 + \alpha^4T_5 + T_4T_5)((1-\alpha)^9 + (1-\alpha)^5 T_4 -(1-\alpha)^4 T_5 - T_4T_5) \\
 &= \left(2^{-9} + 2^{-5}T_4 + 2^{-4}T_5 + T_4T_5\right)\left(2^{-9} + 2^{-5}T_4 - 2^{-4} T_5 - T_4T_5\right) + O(|\alpha - \frac{1}{2}|)\\
&=  \left(2^{-4} + T_4\right)^2\left(2^{-10} - T_5^2 \right) + O(|\alpha - \frac{1}{2}|).
\end{align*}
We showed in (\ref{e:g-hat-infty}) that $||\hat g||_\infty \leq ||g||_\infty/\sqrt{2}$ and so $||\hat g||_\infty \leq \frac{1}{2\sqrt{2}} + O(|\alpha - \frac{1}{2}|)$. Thus here we have
\[T_5 \leq  ||\hat g||_\infty T_4 \leq \frac{1}{2\sqrt{2}} T_4 + O(|\alpha-\frac{1}{2}|)\leq 2^{-6.5} + O(|\alpha - \frac{1}{2}|).\]
Continuing from above, 
\begin{align*}
T_\Phi(f)T_\Phi(1-f) &=  \left(2^{-4} + T_4\right)^2\left(2^{-10} - T_5^2 \right) + O(|\alpha - \frac{1}{2}|) \\
 &\geq \left(2^{-4} + T_4\right)^2\left(2^{-10} - 2^{-3}T_4^2 \right) + O(|\alpha - \frac{1}{2}|)\\
 &= 2^{-18} + 2^{-3}T_4\left(2^{-10} + 2^{-8}T_4 - 2^{-3}T_4^2-T_4^3\right)+ O(|\alpha - \frac{1}{2}|).
 \end{align*}
One checks that the polynomial $q(x)=2^{-10}+2^{-8}x-2^{-3}x^2 -x^3$ is positive for $x \in [0,0.07]$. But $0 \leq T_4 \leq 2^{-5} +O(|\alpha - \frac{1}{2}|)$, and so for $c_2$ chosen small enough we have that $T_4 \in [0,0.07]$. Finally note that $T_4 \geq ||\hat g||_\infty^4$. This completes the proof.
\end{proof}

We note that explicit values for the constants $c_0$, $c_1$, $c_2$, $c_3$ and $C_4$ could be computed if one so wished. As a consequence, one could extract from the following proof an explicit $l_0$ such that for all $l \geq l_0$, $\Phi^{(l)}$ is common. 

\begin{proof}[Proof of Theorem \ref{t:main-alon}]
Recalling (\ref{e:l-free}), we wish to show that for $l$ sufficiently large
\[\alpha^l T_\Phi(f) + (1-\alpha)^lT_\Phi(1-f) \geq 2^{-l-8},\]
where $\alpha := \E f$. Let $g: = f - \alpha$.  We split into three cases: the first is when $|\alpha - \frac{1}{2}|$ is suitably small and $||\hat g||_\infty$ is suitably small; the second is when $|\alpha - \frac{1}{2}|$ is suitably small and $||\hat g||_\infty$ is large; the third is when $|\alpha - \frac{1}{2}|$ is large. In this proof, all asymptotic notation is with respect to $l$.

Firstly, if $\alpha \in [\frac{1}{3},\frac{2}{3}]$ and $||\hat g||_\infty \leq c_1$, then from Lemma \ref{l:local-sid} applied to both $f$ and $1-f$ we have 
\[\alpha^l T(f) + (1-\alpha)^lT(1-f) \geq \alpha^{l+9} + (1-\alpha)^{l+9} \geq 2^{-l-8},\]
where we use convexity for the second inequality.

For the second case, let $c_5$ be a constant depending on $c_1,c_3$ to be determined shortly, and assume that $|\alpha - \frac{1}{2}| \leq \frac{c_5}{\sqrt{l}}$ and that $||\hat g||_\infty \geq c_1$. By the AM-GM inequality, we have 
\[ \alpha^lT(f) + (1-\alpha)^lT(1-f) \geq 2(\alpha(1-\alpha))^{l/2} \sqrt{T(f)T(1-f)}.\]
Invoking Lemma \ref{l:phi-geomcom} by taking $l$ suitably large so $\frac{c_5}{\sqrt{l}}\leq c_2$, we have
\[T(f)T(1-f) \geq 2^{-18} +c_3c_1^4 - \frac{C_4c_5}{\sqrt{l}} \geq 2^{-18} + \frac{1}{2}c_3c_1^4,\]
where we have again taken $l$ suitably large for the second inequality. Continuing from above, there is a positive constant $c_6>0$ which may be computed explicitly in terms $c_1,c_3$ such that  
\begin{align*}
\alpha^lT(f) + (1-\alpha)^lT(1-f) &\geq \left(\alpha(1-\alpha)\right)^{l/2}(2^{-8} + c_6) \\
&= \left(\frac{1}{4}-|\alpha - \frac{1}{2}|^2\right)^{l/2}(2^{-8} + c_6)\\
& \geq \left(1- \frac{4c_5^2}{l}\right)^{l/2}2^{-l}(2^{-8} + c_6)\\
&= \left(1- \frac{4c_5^2}{l}\right)^{l/2}(1+2^8c_6)2^{-l-8}.
\end{align*}
For $c_5$ fixed, $\left(1- \frac{4c_5^2}{l}\right)^{l/2}$ converges to $e^{-2c_5^2}$ as $l\to \infty$. Let $c_5$ be small enough so that $e^{-2c_5^2}(1+2^8c_6)> 1$ and then let $l$ be large enough so that $\left(1- \frac{4c_5^2}{l}\right)^{l/2}(1+2^8c_6)\geq 1$. This completes the proof in this case. 

Finally, assume that $|\alpha - \frac{1}{2}| \geq \frac{c_5}{\sqrt{l}}$. Assume with out loss of generality that $\alpha \geq \frac{1}{2}$. Then we have  by Lemma \ref{l:phi-prev} that
\[\alpha^lT(f) + (1-\alpha)^lT(1-f) \geq \alpha^lT(f) \geq c_0 \left(\frac{1}{2} + \frac{c_5}{\sqrt{l}}\right)^l = 2^{-l}c_0\left(1 + \frac{2c_5}{\sqrt{l}}\right)^{l}.\]
Taking $l$ sufficiently large, this is greater than  $2^{-l-8}$, which completes the proof.
\end{proof}

\section{Geometric commonness}\label{s:geomcom}

In the previous section we obtained a lower bound on the geometric mean of $T(f)$ and $T(1-f)$ in order to conclude that $\Phi$ is Alon. In this section we will introduce the notion of geometric commonness and see in Proposition \ref{p:main} that a bound of this form is in fact necessary. 

\begin{defn}
A system of $t$ linear forms $\Psi$ is \textit{geometrically common} if for all $n\geq 1$ and all $f: \F_p^n \to [0,1]$ with $\E_xf(x) = 1/2$, \[T_\Psi(f)T_\Psi(1-f) \geq 2^{-2t}.\] 
\end{defn}

We recall that a system of equations in $t$ variables is Sidorenko if $T_\Psi(f) \geq (\E_xf(x))^t$ for all $f:\F_p^n \to [0,1]$. It is clear that if $\Psi$ is Sidorenko then it is geometrically common. Furthermore the AM-GM inequality yields that if $\Psi$ is geometrically common then it is common among functions with density $1/2$ (though as we have seen this does not imply commonness in general).

\begin{defn}\label{d:weak-alon}
A system of $t$ linear forms $\Psi$ is \textit{weakly Alon} if for all $\eps>0$, for all sufficiently large $l\geq l_\eps$, for all $n\geq 1$ and for all $f:\F_p^n\to [0,1]$,  
\[T_{\Psi^{(l)}}(f) + T_{\Psi^{(l)}}(1-f)> 2^{-l}\left(2^{1-t}-\eps\right).\]
\end{defn}
Conversely, the condition that a system is not weakly Alon asks that the sequence of systems $\Psi^{(l)}$ is not common infinitely often by a suitably uniform margin. 

\begin{defn}\label{d:prev}
Let $\alpha \in (0,1)$. A system of linear forms $\Psi$ is \textit{$\alpha$-prevalent} if there exists $c(\alpha)>0$ such that for all $n$, for every $f:\F_p^n \to [0,1]$ with $\E f = \alpha$, we have $T_\Psi(f) \geq c(\alpha)$.
\end{defn}
We note that prevalence is monotonic in $\alpha$: if $\alpha < \beta$ and $\Psi$ is $\alpha$-prevalent then it is $\beta$-prevalent.

\begin{remark}\label{r:prev}
By (a generalised) Szemer\'edi's theorem, and Varnavides' strengthening thereof, all translation-invariant systems are $\alpha$-prevalent for all $\alpha \in (0,1)$. It is not difficult to see that the converse holds too, so $\alpha$-prevalence for all $\alpha \in (0,1)$ is equivalent to translation-invariance. Schur triples are $(\frac{1}{3} + \eps)$-prevalent but not $\frac{1}{3}$-prevalent if $p=3$; $(\frac{2}{5} + \eps)$-prevalent but not $\frac{2}{5}$-prevalent if $p=5$; $(\frac{2}{7} + \eps)$-prevalent but not $\frac{2}{7}$-prevalent when $p=7$. See \cite{GR05}, \cite{G05} for the results pertaining to Schur triples. We showed in Lemma \ref{l:phi-prev} that $\Phi$ is $0.45$-prevalent for odd $p$. On the other hand, if $p=3$, then $\Phi$ is not $\frac{1}{3}$-prevalent since it is not translation-invariant. 
\end{remark}

We will only prove the reverse implication of the below Proposition. The forward direction is essentially a union of the arguments in the second and third cases of the proof of Theorem \ref{t:main-alon} and so we omit the details.

\begin{prop}\label{p:main}
A $\frac{1}{2}$-prevalent system of linear forms $\Psi$ is geometrically common if and only if it is weakly Alon.
\end{prop}
\begin{proof}
Suppose $\Psi$ is not geometrically common and let $n$ be such that there exists $f:\F_p^n \to [0,1]$ with $\E f = 1/2$ and $T(f)T(1-f) < 2^{-2t}$. If $T(f)=T(1-f)$ then we have immediately that $T(f) + T(1-f) < 2^{-t+1}$, and indeed there is $\eps > 0$ such that $T(f) + T(1-f) \leq 2^{-t+1} - \eps$. Then by (\ref{e:l-free}), 
\[T_{\Psi^{(l)}}(f) + T_{\Psi^{(l)}}(1-f) = 2^{-l}\left(T(f) + T(1-f)\right) \leq 2^{-l}(2^{1-t}-\eps), \] for all $l\geq 1$. Thus $\Psi$ is not weakly Alon in this case. 

Assume by symmetry that $T(1-f)> T(f)$. Define $c := \log\sqrt{T(1-f)/T(f)}~>~0$. By the $\frac{1}{2}$-prevalence of $\Psi$, we have that $c$ is bounded above uniformly in $n, f$. Now let $l\geq \frac{45c}{4}$ be a large integer to be determined more precisely later. Let $S = \{x:f(x) \leq \frac{9}{10}\}$; since $\E f = 1/2$, we have by Markov's inequality that $|S| \geq \frac{4p^n}{9}$. Let $g:\F_p^n \to [0,1]$ take value $\frac{p^n}{|S|}\cdot \frac{c}{2l} \leq \frac{1}{10}$ on $S$ and be zero otherwise. Then we have that $\E_xg(x) = \frac{c}{2l}$ and that $f + g$ takes values in $[0,1]$.  We claim that $f+g$ exhibits that $\Psi$ is not weakly Alon.

Note that for functions $f_i : \F_p^n \to [0,1]$ and nonzero linear forms $(\psi_i)_{i=1}^t$ we have that $\E_{\vect x} f_1(\psi_1(\vect x))\cdots f_t(\psi_t(\vect x)) \leq \E_x f_i(x)$ for any $i$, and so by the multilinearity of the $T$ operator we have that $T(f+g) = T(f) + O(1/l)$. Similarly, $T(1-(f+g)) = T(1-f) + O(1/l)$. Thus, 
\begin{align*}
T_{\Psi^{(l)}}(f+g) + T_{\Psi^{(l)}}(1-(f+g)) &= \left( \frac{1}{2} + \frac{c}{2l} \right)^l T(f+g) + \left(\frac{1}{2} - \frac{c}{2l}\right)^l T(1-(f+g))\\ 
&= \frac{1}{2^l}\left((1+c/l)^lT(f) + (1-c/l)^lT(1-f) + O(1/l)\right) \\
&= \frac{1}{2^l}\left(e^cT(f) + e^{-c}T(1-f) + O(1/l)\right) \\ 
&= \frac{1}{2^l}\left(2\sqrt{T(f)T(1-f)} + O(1/l)\right).
\end{align*}
Now we have that $\sqrt{T(f)T(1-f)} < 1/2^{2t}$ and so there exists $\eps>0$ such that for $l$ sufficiently large we have $T_{\Psi^{(l)}}(f+g) + T_{\Psi^{(l)}}(1-(f+g)) < 2^{-l}(2^{1-t}-\eps)$. Thus we conclude that $\Psi$ is not weakly Alon.

The proof of the reverse implication is omitted. 
\end{proof}

\begin{remark}
It is clear that if $\Psi$ is not $\frac{1}{2}$-prevalent then it is not geometrically common. We have not investigated whether the (weak) Alon property implies $\frac{1}{2}$-prevalence.
\end{remark}

We conclude with a picture showing the relationships between the various notions discussed in this document.

\begin{figure}\label{f:notions}
\centering
\includegraphics[scale = 0.6]{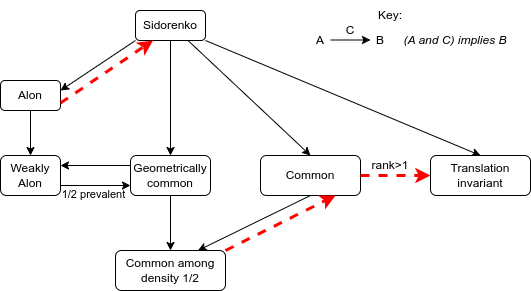}
\caption{Implications between various notions. Red arrows, from left to right, correspond to answers to Questions \ref{q:alon}, \ref{q:klm6}, \ref{q:klm5}, which show that the corresponding implications do not hold.}
\end{figure}

\bibliographystyle{alpha}
\bibliography{sid2}

\end{document}